\documentclass[10pt]{scrartcl}
\usepackage{amsmath, amssymb, amsthm}
\usepackage{graphicx}
\usepackage{subfigure}
\usepackage{color}
\usepackage{cite}

\newcommand{\be}{\begin{equation}}
\newcommand{\ee}{\end{equation}}

\theoremstyle{plain}
\newtheorem{thm}{Theorem}[section]

\newtheorem{prop}{Proposition}[section]
\newtheorem{cor}[prop]{Corollary}

\newtheorem{rem}{Remark}[section]
\newtheorem{example}{Example}[section]
\newtheorem{definition}{Definition}[section]

\graphicspath{{images/}}

\author{Undine Leopold, Horst Martini}
\title{Geometry of simplices in Minkowski spaces}
\begin{document}
\maketitle

\begin{abstract}
There are many problems and configurations in Euclidean geometry that were never
extended to the framework of (normed or) finite dimensional real Banach spaces,
although their original versions are inspiring for this type of generalization, and the
analogous definitions for normed spaces represent a promising topic. An example is
the geometry of simplices in non-Euclidean normed spaces. We present new generalizations
of well known properties of Euclidean simplices. These results refer to analogues of
circumcenters, Euler lines, and Feuerbach spheres of simplices
in normed spaces. Using duality, we also get natural theorems on angular bisectors as well as
in- and exspheres of (dual) simplices.
\end{abstract}
\section{Introduction}\label{sec:intro}

Looking at basic literature on the geometry of finite dimensional real Banach spaces (also called Minkowski spaces; see, e.g., the monograph \cite{t1996:mg} and the surveys \cite{msw2001:tgomsas1} and \cite{ms2004:tgomsas2}), the reader will observe that there exists no systematic or satisfying representation of results in the direction of (also higher dimensional) Elementary Geometry in such spaces. In other words, Elementary Geometry of Minkowski spaces is not really a developed field. 

Although there are well known theorems in Euclidean geometry which, in many cases, are conveniently extendable to Minkowski spaces, until now this was not systematically done. An example is the geometry of simplices in non-Euclidean Minkowski spaces - interesting properties certainly wait for their discovery! Inspired by this lack of results, we derive new extensions (Minkowskian analogues) of well known properties of Euclidean simplices.  These results refer especially (but not only) to circumspheres and circumcenters, allowing us to define a weak notion of regularity. Via duality, we also get natural theorems on angular bisectors as well as in- and exspheres of (dual) simplices.

A $d$-dimensional (\emph{normed} or) \emph{Minkowski space} $(\mathbb{R}^d,\|\cdot\|)$ is the vector space $\mathbb{R}^d$ equipped with a norm $\| \cdot \|$. A norm can be given implicitly by its \emph{unit ball}, which is a convex body centered at the origin $o$; its boundary is the \emph{unit sphere} of the normed space. Any homothet of the unit ball is called a \emph{Minkowskian ball} and denoted by $B(X,r)$, where $X$ is its center and $r>0$ its radius, while its boundary is the \emph{Minkowskian sphere} $S(X,r)$. Two-dimensional Minkowski spaces are \emph{Minkowski planes}, and for an overview on what has been done in the geometry of normed planes and spaces we refer again to \cite{t1996:mg}, \cite{msw2001:tgomsas1}, and \cite{ms2004:tgomsas2}.

The fundamental difference between non-Euclidean Minkowski spaces and the Euclidean space is the absence of an inner product, and thus the notions of angle and orthogonality do not exist in the usual sense. Nevertheless, several \emph{types of orthogonality} can be defined (see \cite{ab1988:oinlsas1},\cite{ab1989:oinlsas2}, and \cite{amw2012:oboaioinls} for an overview; for angles we refer to \cite{bhmt2017:ains}), with \emph{isosceles} and \emph{Birkhoff orthogonalities} being the most prominent examples. We say that $y$ is \emph{Birkhoff orthogonal} to $x$, denoted $x\perp_B y$, when $\|x\|\leq \|x+\alpha y\|$ for any $\alpha \in \mathbb{R}$. 

For any two distinct points $P$, $Q$, we denote by $[PQ]$ the \emph{closed segment}, by $\langle PQ \rangle$ the \emph{spanned line} (affine hull), and by $[PQ\rangle$ the \emph{ray} $\{P+\lambda (Q-P)\ \vert\ \lambda \geq 0\}$; we write $\|[PQ]\|:=\|Q-P\|$ for the \emph{length} of $[PQ]$.We will use the common abbreviations $\mathrm{aff}$, $\mathrm{conv}$, $\partial$, and $\mathrm{cone} $ for the affine hull, convex hull, boundary and cone over a set, respectively.

In this article, we focus on the geometry of simplices in Minkowski spaces. As usual, a \emph{$d$-simplex} is the convex hull of $d+1$ points in general linear position, or the non-empty intersection of $d+1$ closed half-spaces in general position.

In the last section, we discuss \emph{dual} statements. For this purpose, the dual space of linear functionals on $(\mathbb{R}^d,\|\cdot\|)$, a $d$-dimensional real vector space in its own right, is identified with the original Minkowski space by fixing bases and using an auxiliary Euclidean structure.

\section{Circumcenters}\label{sec:circumcenter}

Before presenting our results on \emph{circumcenters of simplices}, we underline that we mean the \emph{centers of circumspheres} (or \emph{-balls}) \emph{of simplices}, i.e., of Minkowskian spheres containing all the vertices of the respective simplex (see, e.g., \cite{ams2012:medcacinpp1}). A related, but different notion is that of \emph{minimal enclosing spheres of simplices}, sometimes also called circumspheres (cf., e.g., \cite{ams2012:medcacinpp2}); this notion is not discussed here. In the two-dimensional situation, circumspheres and -balls are called \emph{circumcircles} and \emph{-discs}.
In Minkowski spaces, simplices may have several, precisely one, or no circumcenter at all, depending on the shape of the unit ball, see Figure \ref{fig:1}. Examples without circumcenters may only be constructed for non-smooth norms, as all smooth norms allow inscription into a ball \cite{g1969:sioah,m1987:tdoamispocg}. \emph{We focus on the case where there is at least one circumcenter.}

\begin{figure}
\begin{center}
\includegraphics[width=.7\textwidth]{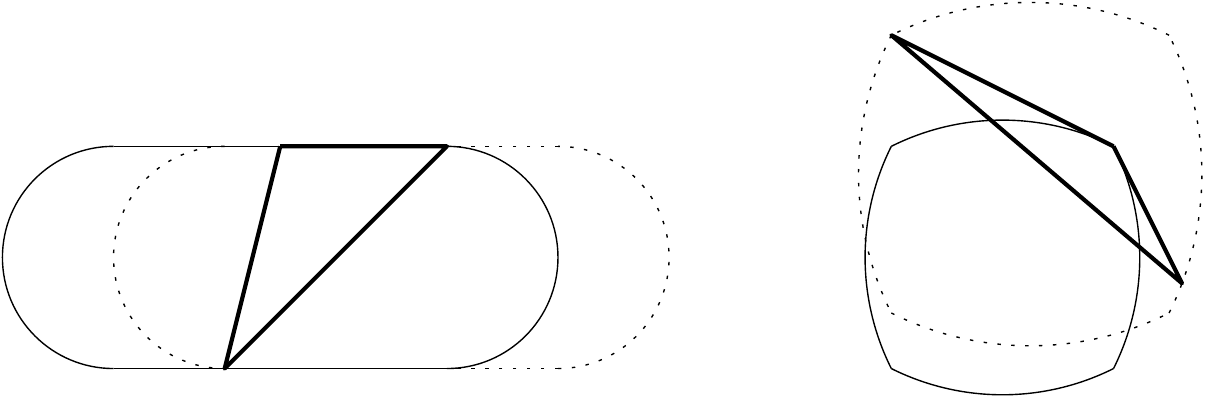}
\end{center}
\caption{On the left: A triangle with several circumcenters, as illustrated by suitable translates of the unit ball. The triangle on the right does not possess any circumcenter with respect to the indicated norm.}
\label{fig:1}
\end{figure}

Inspired by the work \cite{ams2012:medcacinpp1} of Alonso, Martini and Spirova for the planar case (see also \cite{v2017:oocinp}), we establish results about the location of circumcenters in the $d$-dimensional case, i.e., about the \emph{centers of circumspheres} as defined above. Incidentally, this also proves some statements about the planar case in a more condensed form than in \cite{ams2012:medcacinpp1}. We define the \emph{medial hyperplane} between a facet $F$ of a $d$-dimensional simplex and the opposite vertex $A$ as the hyperplane bisecting every segment from $A$ to a point in $F$.

\begin{thm}\label{thm:2.1}
Let $T$ be a Minkowskian $d$-simplex with circumball $B$ and circumsphere $S=\partial B$ centered at $M$. Let $F$ be an arbitrary facet of $T$, let $A$ be the opposite vertex, and let $E$ be the medial hyperplane  between $F$ and $A$. The following are equivalent.\\
a) $M$ lies in the same open halfspace as $A$ with respect to $E$.\\
b) $M\notin \mathrm{cone}((\mathrm{aff}(F)\cap B),A):=\{A+\lambda((\mathrm{aff}(F)\cap B)-A),\lambda\geq 0\}$.\\
\end{thm}

\begin{proof}
Let $D$ be the point of intersection of the ray $[AM\rangle$ with $\mathrm{aff}(F)$. Now, $M$ lies in the same open halfspace as $A$ with respect to $E$ if and only if $\|[MD]\|>\|[AM]\|$. Furthermore $\|[MD]\|>\|[AM]\|$ if and only if $D$ is outside of $B$, i.e., if and only if $D$ is not in $\mathrm{aff}(F)\cap B$. Finally, $D$ is not in $\mathrm{aff}(F)\cap B$ if and only if $M$ lies outside of $\mathrm{cone}((\mathrm{aff}(F)\cap B),A)$, which proves the theorem.
\end{proof}

For two-dimensional simplices, we are able to extract additional information. 

\begin{thm}\label{thm:2.2}
Let $T$ be a triangle in a Minkowski plane, with circumball $B$ and circumcircle $S=\partial B$ centered at $M$. Let $F$ be an arbitrary side of $T$ opposite to a vertex $A$, and let $E$ be the medial line bisecting every segment between $F$ and $A$. Suppose $M$ lies in $\mathrm{cone}((\mathrm{aff}(F)\cap B)\setminus F,A)$. Then $M$ lies in $E$.
\end{thm}

\begin{proof} First, we note that if $(\mathrm{aff}(F)\cap B)\setminus F\neq \emptyset$ in the plane, then the planar circumball $B$ is not strictly convex, since at least $3$ points on the line $\mathrm{aff}(F)$, namely the vertices of $F$ and at least one other point in $(\mathrm{aff}(F)\cap B)\setminus F$, are part of the circle $S=\partial B$, leading to $(\mathrm{aff}(F)\cap B)\subset S$.
Thus, if $M$ lies in $\mathrm{cone}((\mathrm{aff}(F)\cap B)\setminus F,A)$ and $D$ is the point of intersection of the ray $[AM\rangle$ with $\mathrm{aff}(F)$, then $\|[MD]\|=\|[AM]\|$, and $M$ lies in $E$.
\end{proof}

Combining Theorem \ref{thm:2.1} and  Theorem \ref{thm:2.2}, we obtain that the circumcenter of any $2$-simplex $ABC$ in a Minkowski plane can only lie in the interior or boundary of the shaded region in Figure \ref{fig:2} which is defined by the lines through edge midpoints, as established by other (longer) arguments in \cite[Theorem 4.1.]{ams2012:medcacinpp1}. In fact, in \cite{ams2012:medcacinpp1} the stronger statement is proved that for any point $M$ in Figure \ref{fig:2} a norm exists which makes $M$ a circumcenter of the $2$-simplex $ABC$ if and only if $M$ lies in the shaded region. It is not hard to see why the converse in this statement is true; the convex hull of triangle $ABC$ and its image $A'B'C'$ with respect to a half-turn around a point $M$ located in the shaded region is a suitable unit ball (the antipodal pairs $A$, $A'$, $B$, $B'$, and $C$, $C'$ are contained in its boundary). 

\begin{figure}
\begin{center}
\includegraphics[width=.5\textwidth]{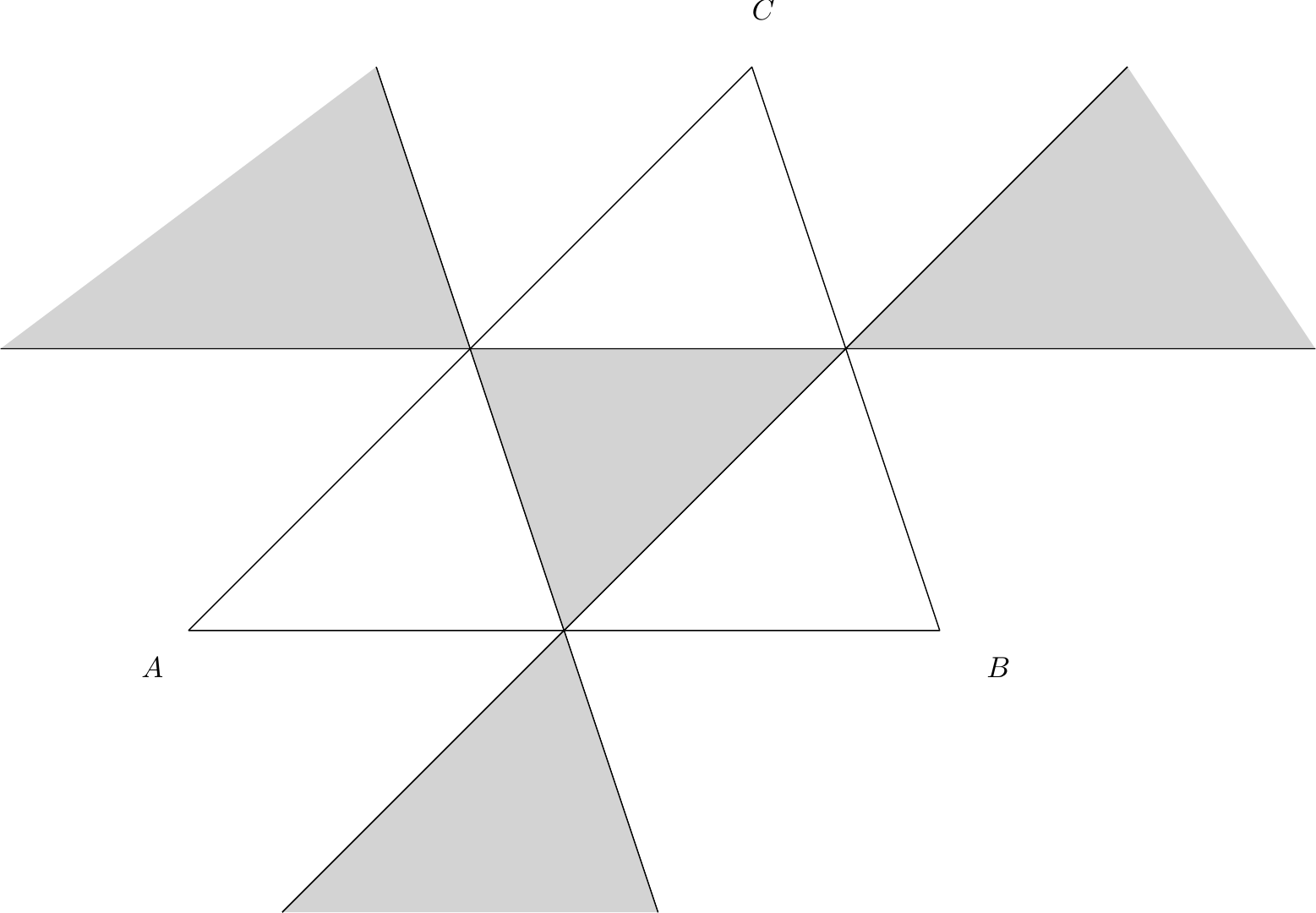}
\end{center}
\caption{Circumcenters of the triangle $ABC$ can only be located in the shaded regions determined by the edge midpoints.}
\label{fig:2}
\end{figure}

The \emph{medial polytope} of a $d$-simplex is its complete truncation (the truncation of the simplex up to the midpoints of edges). Some of its facets lie in the truncating medial hyperplanes, whereas the remaining facets lie in the hyperplanes supporting the original simplex for $d\geq 3$ (these facets are truncated facets of the original simplex). The following Corollary is an immediate consequence of Theorem \ref{thm:2.1}.

\begin{cor}\label{cor:2.1}
Let $T$ be a Minkowskian $d$-simplex ($d\geq 2$). If a circumcenter $M$ of $T$ lies within $T$, then $M$ lies within the medial polytope of $T$. 
\end{cor}

The point
\[G:=\frac{\sum_{i=0}^d A_i}{d+1}\]
is the \emph{centroid} of a $d$-simplex $T$ with vertices $A_0,\ldots,A_{d}$. It is the point of intersection of the \emph{medians} $[A_iA'_i]$, where \[A'_i:= \frac{\sum_{\substack{j=0\\j\neq i}}^d A_j}{d}\] is the centroid of the facet opposite $A_i$. The well-known fact that $\|[A_iG]\|\colon\|[GA'_i]\|=d:1$ for any $i=0,\ldots,d$ proves that $G$ lies in the relative interior of the medial polytope for $d\geq 1$ (for $d=0$, everything collapses to just one point, so the statement is still true but trivial). 

For $2$-dimensional simplices, the following is a consequence of \cite[Theorem 4.2. (b)]{ams2012:medcacinpp1}. We also show an independent proof.

\begin{thm}\label{thm:2.3}
In dimension $d=2$, if a triangle $T$ possesses a circumcenter $M$ in the relative interior of its medial polygon (triangle), then $M$ is the unique circumcenter of~$T$. In particular, a triangle $T$ whose centroid $G$ is a circumcenter possesses no other circumcenter besides~$G$. 
\end{thm}

\begin{proof}
Let $T$ be a triangle in a Minkowski plane. Assume that a point $M_1$ in the relative interior of the medial triangle of $T$ is a circumcenter of $T$, and assume that there exists a circumcenter $M_2$ distinct from $M_1$. Denote the corresponding (positive) radii of the circumcircles by $r_1$ and $r_2$. Then $T$ and $T'=M_1+\frac{r_1}{r_2}(T-M_1)$ are homothets with homothety center $X$ on the line $\langle M_2M_1\rangle$ such that $\|[XM_2]\|\colon\|[XM_1]\|=r_2\colon r_1$ if $r_1\neq r_2$, and with homothety center $\infty$ if $r_1=r_2$. Furthermore, the vertices of $T$ and $T'$ lie on the circle $S(M_1,r_1)$.

The vertices of $T$ and the vertices of $T'$ all lie on the boundary of the convex set $\mathrm{conv}(T\cup T')$ and, since $T\neq T'$, at most two of these vertices coincide.  Every vertex of $T$ and its corresponding (by homothety) vertex of $T'$ define, if they are distinct, a supporting line of  $\mathrm{conv}(T\cup T')$. All so defined supporting lines also run through $X$. 

There are precisely two extremal lines of the infinite cone over $T$ with apex $X$, respectively of the degenerate cone (strip) for $X=\infty$. They are the only supporting lines which contain a point of the (degenerate) cone distinct from $X$. 
Since there now are at least three vertices 
on one of these supporting lines, say $\ell$, and those vertices must also lie on the sphere  $S(M_1,r_1)$, we conclude that $\ell$ is a supporting line of $S(M_1,r_1)$ containing at least one side $F$ of $T$.

Now, since $M_1$ lies in the relative interior of the medial triangle of $T$, $M_1$ and $F$ lie in the same open half-plane with respect to the medial hyperplane $E$ between $\ell$ and the opposite vertex $A$ of $T$, and $M_1$ lies in the cone over $F$ with apex $A$. Then $|[AM_1]|>r_1$. Thus
$M_1$ cannot be a circumcenter, which shows that our initial assumption (there exists a circumcenter besides $M_1$) was false. The second statement follows immediately from the fact that $G$ lies inside the medial triangle of $T$. 
\end{proof}

We have noted, on the one hand, that simplices may not have any circumcenters, but only if the unit ball is not smooth. On the other hand, we point out that, for $d>2$, there are strictly convex (smooth or nonsmooth) norms for which certain simplices may have more than one circumcenter. This is due to the fact that the intersection of two spheres with different centers (and possibly different radii) need not be contained in a hyperplane. 

\begin{example}\label{ex:2.1}
Consider the $3$-simplices $ABCD$ and $A'B'C'D'$ in an auxiliary Euclidean $3$-space, as in Figure~\ref{fig:3}. They are homothetic with respect to point $X$, with ratio
\[r:=\frac{\|XA\|}{\|XA'\|}=\frac{\|XB\|}{\|XB'\|}=\frac{\|XC\|}{\|XC'\|}=\frac{\|XD\|}{\|XD'\|}>1.\]
Suppose the segments $[AA']$, $[BB']$, $[CC']$, $[DD']$ lie on exposed rays of \[\mathrm{cone}(ABCD,X):=\{X+\lambda(\mathrm{conv}\{A, B, C, D\}-X),\lambda\geq 0\}\] and contain no other points of the simplices besides the endpoints. Fix the origin~$o$ on an interior ray of the above cone, such that $-\mathrm{conv}\{A, B, C, D\}$ (and, therefore, $-\mathrm{conv}\{A', B', C', D'\}$) lies in the interior of the cone; this is always possible for $o$ ``sufficiently far away'' from $X$. It is then possible to circumscribe a smooth, centrally symmetric (with respect to $o$) and strictly convex body $K$ around both simplices. Moreover, the homothet $X+r(K-X)$ of $K$ is also circumscribed around simplex $ABCD$, with center $X+r(o-X)$. Thus, simplex $ABCD$ has at least two circumcenters in the smooth, strictly convex norm induced by taking $K$ as the unit ball. 
\end{example}

\begin{figure}
\begin{center}
\includegraphics[width=.8\textwidth]{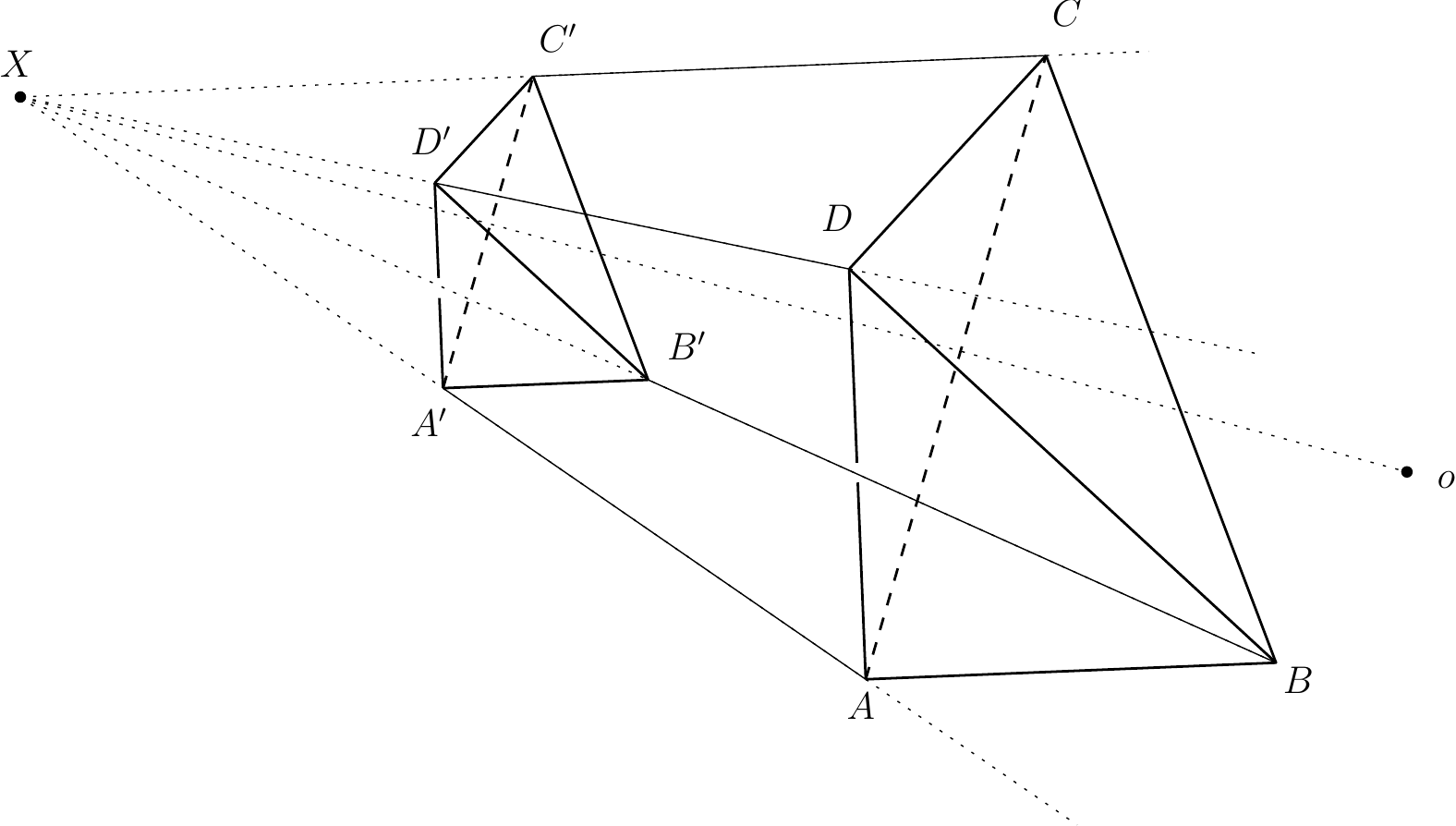}
\end{center}
\caption{Simplex $ABCD$ possesses several circumcenters in a smooth, strictly convex norm, since homothetic simplices $ABCD$ and $A'B'C'D'$ can be simultaneously inscribed in a smooth, strictly convex body which is centrally symmetric with respect to $o$.}
\label{fig:3}
\end{figure}

Examples can also be constructed if the distance functional is more general, i.e., for \emph{gauges} whose unit balls need not be centered at the origin. For strictly convex norms in $d=2$, however, strict convexity prevents several circumcenters; the unit circle necessarily contains a segment if a triangle possesses more than one circumcircle, see the discussion and references presented in \cite[Section 7.1.]{msw2001:tgomsas1}, and also \cite[Corollary 3.1.]{ams2012:medcacinpp1}. Note, moreover, that Theorem \ref{thm:2.3} does not generalize to arbitrary dimensions $d>2$, as shown by the next example.

\begin{example}\label{ex:2.2}
Define a norm by using a parallelepiped $Q$ centered at the origin of an auxiliary Euclidean $3$-space as the unit ball. We construct a simplex $\mathrm{conv}\{A,B,C,D\}$ with several circumcenters in this new norm, one of which is its centroid. Select point $A$ as the midpoint of an edge $e$ of $Q$. Let $E_A$ be the plane through $A$ and the midpoints of all other edges of $Q$ which are parallel to $e$. Select one of the faces incident at $e$ as the ``top'' of $Q$ and intersect the parallelepiped with a plane $E$ at $\frac{2}{3}$ of the distance between the ``top'' of $Q$ and its opposite face. Select $B$ as one of the two points in $E\cap E_A\cap \partial Q$, such that $B$ and $A$ lie in the same face of $Q$. Place the two remaining vertices $C$, $D$ in the relative interior of the face opposite $B$ and in $E\cap \partial Q$, such that $E_A$ bisects the line segment connecting them. Then the simplex generated as the convex hull of the four vertices $A$, $B$, $C$, $D$ possesses several circumballs (e.g., $Q$ and small translations of $Q$ parallel to $CD$). Moreover, the origin $o$ is not only a circumcenter of $\mathrm{conv}\{A,B,C,D\}$, but also the centroid of this simplex. See Figure \ref{fig:4}. 
\end{example}

\begin{figure}
\begin{center}
\includegraphics[width=.5\textwidth]{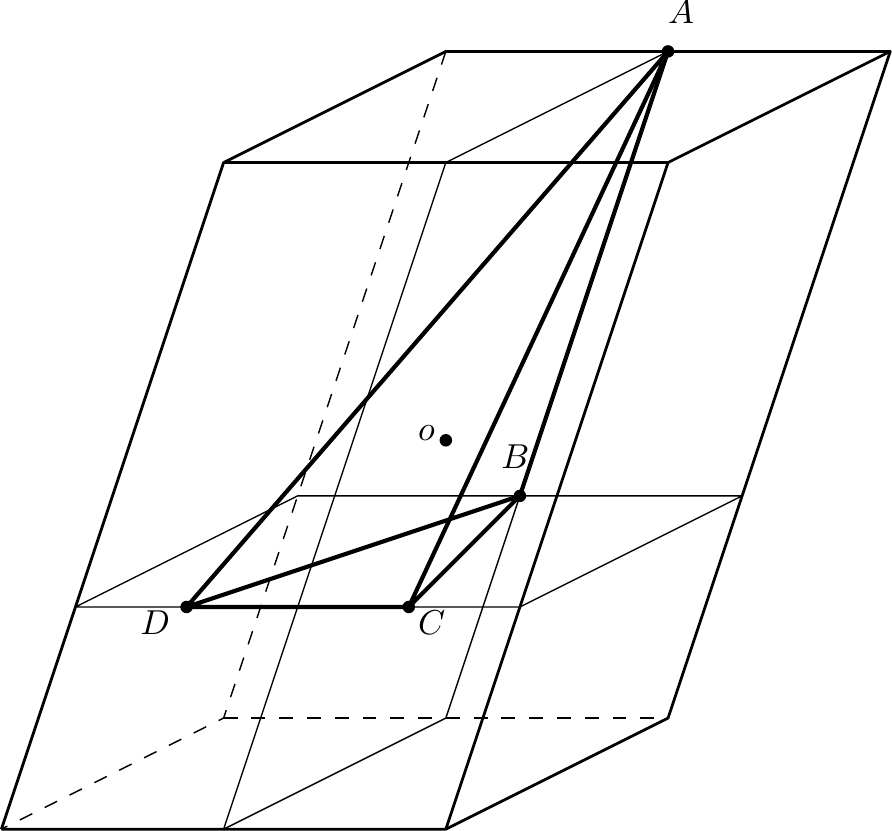}
\end{center}
\caption{Simplex $ABCD$ has several circumcenters, one of which is its centroid (the origin $o$).}
\label{fig:4}
\end{figure}

\section{Quasiregular simplices}\label{sec:quasiregular}

Simplices where the centroid is a circumcenter are special in that \emph{all medians of such a simplex have equal length}. In dimension $d=2$ (see Theorem \ref{thm:2.3}) such a circumcenter is necessarily unique, and in the Euclidean plane, coincidence of \emph{the} circumcenter and centroid implies regularity. This warrants a definition.

\begin{definition}\label{def:3.1}
Let $T$ be a $d$-simplex in Minkowskian $d$-space possessing a circumcenter. We say $T$ is \emph{$AG$-quasiregular} if its centroid is a circumcenter.
\end{definition}

\begin{rem}\label{rem:3.1}
$AG$-quasiregularity for $d=2$ is the same as the notion of $AG$-regularity for triangles in strictly convex normed planes, defined by Martini and Spirova in \cite{ms2011:rtinp} and meaning \emph{regular in the sense of Asplund and Gr\"unbaum}, see \cite{ag1960:otgomp}. In higher dimensional Euclidean space, however, the coincidence of centroid and circumcenter does not by itself imply regularity (see \cite[Theorem 3.2.]{ehm2005:coscarfs}). More precisely, it implies that the simplex belongs to the larger family of $d$-simplices with the following property: for all their facets the sum of squares of all their ${d \choose 2}$ edge-lengths is equal. This is why, also in the more general Minkowskian case, we speak of quasiregularity, as a Minkowskian analogue of a weaker kind of regularity.
\end{rem}

In \cite{lm2017:mpelafsims} the \emph{Euler line} and the \emph{Feuerbach-} or \emph{$2(d+1)$-sphere} associated to a circumcenter $M$ of a $d$-simplex $T$ have been studied. The Euler line is the straight line containing $M$, the centroid $G$, the \emph{associated Feuerbach center $F_M$} (the center of the aforementioned sphere), the \emph{associated Monge point $N_M$}, and and the so-called \emph{associated complementary point $P_M$}. We recall from \cite{lm2017:mpelafsims} that the point $N_M$ is the point of concurrence of certain lines, namely those which are well-defined by the centroid of a $(d-2)$-face (or ridge) of $T$ and the midpoint of the opposite edge. Furthermore, the point $P_M$ is the point of concurrence of certain other lines, each containing a vertex of $T$ and parallel to a line defined by the centroid of the opposite facet and $M$, provided the latter two points do not coincide.
Then, the Feuerbach center $F_M$ divides $[MP_M]$ internally in the ratio $1:(d-1)$. It is the center of a sphere analogous to the well-known Feuerbach circle of a triangle in the Euclidean plane. 
For $AG$-quasiregular simplices the following holds.

\begin{prop}\label{prop:3.1}
In an $AG$-quasiregular $d$-simplex in Minkowskian $d$-space, the centroid coincides with a circumcenter and the associated Monge point, Feuerbach center, and complementary point, i.e., the associated Euler line collapses to a single point. In particular, the centroid is both a circumcenter and the center of a sphere through the centroids of the facets, the \emph{Feuerbach-sphere}. $AG$-quasiregular $d$-simplices are the only simplices which have the property that one associated Euler line collapses.
\end{prop}

Recall that an $AG$-quasiregular simplex in a $d$-dimensional Minkowski space (for $d\geq 3$) may have other circumcenters besides the centroid (see Example \ref{ex:2.2}), and each such other circumcenter defines an Euler line together with the centroid. Note, however, that there also exist simplices with a unique circumcenter which are \emph{not} $AG$-quasiregular, i.e., where the unique circumcenter does not coincide with the centroid; for example, this is true for all simplices except the regular ones in the Euclidean norm. 

It is natural to ask whether $AG$-quasiregular simplices exist often in any dimension $d\geq 2$. Given a Minkowskian $(d-1)$-sphere $S(o,1)$ (w.l.o.g. a unit sphere centered at $o$) and a point $P_0$ on such a sphere, does there exist an $AG$-quasiregular simplex inscribed into that sphere with  vertex $P_0$? Martini and Spirova \cite{ms2011:rtinp} examined this question in strictly convex normed planes and proved that there is a unique such simplex (then called $AG$-regular triangle). We now demonstrate that the answer in higher dimensions is also positive, although not unique.

\begin{thm}\label{thm:3.1}
Let $P_0$ be a point on the Minkowskian $(d-1)$-sphere $S(o,1)$ in an arbitrary Minkowski space ($d\geq 2$). Then there exists an $AG$-quasiregular $d$-simplex $T$ with $P_0$ as one of its vertices and $S(o,1)$ as circumsphere.
\end{thm}

\begin{proof}
The proof is by recursive construction of an example. At each step of the recursion, we have a convex body $B$ of dimension $l$ with a (relative) interior point $G$, and $G$ is the prescribed centroid for (as yet to be determined) $l+1$ points on the boundary $\partial B$. Unless otherwise stated, we operate relative to the affine hull of $B$, which has dimension $l$, and use the induced norm (in fact, we only use division ratios of collinear points). The construction of a $d$-simplex consists of two phases, starting with $G:=o$ and $B:=B(o,1)$:
\begin{paragraph}{Phase 1, $l>2$.}
Select a chord of $B$ passing through $G$. In the case that $l=d$, we select the chord containing $P_0$. Label its two endpoints $P$ and $Q$, such that $\|[PG]\|\leq\|[QG]\|$, and select $P$ as a vertex of the simplex (for $l=d$, select $P:=P_0$ to be the prescribed vertex). Determine $G':=G+\frac{G-P}{l}$ which lies strictly between $G$ and $Q$ and therefore in the (relative) interior of $B$. Select an affine $(l-1)$-plane $E$ through $G'$ but not through $G$ (and therefore not through $P$, $Q$) and designate it to be the affine hull of the remaining $l$ vertices ($E$, by construction, contains none of the previously selected vertices; this is to ensure general position). Then set $l:=l-1$, $B:=B\cap E$, $G:=G'$ and proceed to the next step of the recursion.
\end{paragraph}
\begin{paragraph}{Phase 2, $l=2$.}
Select a chord of $B$ passing through $G$, such that it is not divided by $G$ in the ratio $1:2$ (in the case that $l=d=2$, we select the chord containing $P_0$). This is always possible for reasons of continuity. Label its two endpoints $P$ and $Q$, such that $\|[PG]\|\leq\|[QG]\|$, and select $P$ as a vertex (for $l=d=2$, select $P:=P_0$). Determine $G':=G+\frac{G-P}{2}$. For the remaining two vertices, select another chord of $B$, this time through $G'$, such that $G'$ is its midpoint (which is again possible for reasons of continuity). Observe that the chord must be distinct from $[PQ]$, as $[PQ]$ was chosen such that $G'$ would not be its midpoint. Select the endpoints $R$, $S$ of the second chord as the remaining vertices.
\end{paragraph}

All $d+1$ chosen vertices, by construction, lie on the $(d-1)$-sphere $S(o,1)$ and are in general position because they are affinely independent. Moreover, their centroid by construction is $o$. Thus, we have constructed an $AG$-quasiregular $d$-simplex with prescribed vertex $P_0$ and $o$ as centroid and circumcenter.
\end{proof}

\begin{rem}\label{rem:3.2}
Theorem \ref{thm:3.1} prescribes only one vertex arbitrarily. In general, we cannot hope to prescribe more than half of the vertices. For example, note that we cannot prescribe a $(d-2)$-face for dimension $d\geq 4$; let $G_{3}$ be the centroid of three points on a Minkowskian unit $3$-sphere in $4$-space, then $\|G_{3}\|\leq 1$, but it may very well be that $\|-\frac{3}{2}G_{3}\|> 1$.\end{rem}

\begin{rem}\label{rem:3.3}
The case $d=2$ leaves out Phase $1$ completely and, for a strictly convex norm, determines the unique $AG$-regular triangle with one prescribed vertex, as in \cite{ms2011:rtinp} (the uniqueness follows from strict convexity). For $d\geq 3$, uniqueness of the construction in the case of a strictly convex norm (or even a strictly convex, smooth norm) fails even when we endeavor to prescribe half of the vertices. E.g., for $d=3$, select $2$ vertices on the unit sphere in Euclidean $3$-space, and determine the centroid $G'_2$ of the missing $2$ vertices. There is now much freedom in the selection of these vertices, as all antipodes (diametrally opposite points) on the circle $S(o,1)\cap E$  in the plane $E$ through $G'_2$ and perpendicular to $\langle oG'_2\rangle$ are suitable, except those where all four vertices become coplanar (and the simplex collapses).
\end{rem}

\section{Duality}\label{sec:duality}

In \cite{a2003:otgosims} Averkov examined further properties of simplices in Minkowski spaces. We are particularly interested in extensions of the statements regarding duality which, interestingly, build a connection with the notion of \emph{insphere} and \emph{incenter}. For this purpose, consider the Minkowskian $d$-space $\mathbb{M}^d:=(\mathbb{R}^d,\|\cdot\|)$ with unit ball $B$, and the dual space $(\mathbb{M}^d)^*$ of linear functionals on $\mathbb{M}^d$. The \emph{dual norm} on $(\mathbb{M}^d)^*$ is the support function on $B$, i.e.,
\[\|f\|^*=\mathrm{sup} \{f(x)\colon x\in B\}\] with \emph{dual unit ball}
\[B^*=\{f\in (\mathbb{M}^d)^*\colon \|f\|^*\leq 1\}.\] Using an auxiliary Euclidean metric induced by an inner product $\langle\cdot,\cdot\rangle$, $\mathbb{M}^d$ and $(\mathbb{M}^d)^*$ can be identified, with $B^*$ as the \emph{polar reciprocal} of $B$ with respect to the inner product. That is,
\[B^*=\{y\in\mathbb{M}^d \colon \langle x,y\rangle\leq 1 \text{ for all } x \in B\}.\] We may then also denote a normed space with a certain unit ball as $\mathbb{M}^d(B)$ and the dual (identified) space (with the dual norm) as $\mathbb{M}^d(B^*)$. Recall that the unit ball for $\mathbb{M}^d(B^{**})$ is a homothet of the unit ball for $\mathbb{M}^d(B)$.
The \emph{dual body $K^*_P$} of a convex body $K$ with respect to a point $P$ is simply its polar reciprocal with respect to the auxiliary Euclidean structure, and with $P$ taking the role of the origin.

Recall that the \emph{Minkowskian distance} of a point $P$ to a hyperplane $H$ is the radius of a smallest Minkowskian ball centered at $P$ such that $H$ is tangent to this ball. Analogously, the \emph{Minkowskian height} of a $d$-simplex $T$ with respect to a vertex $A$ is the distance of $A$ to the supporting hyperplane containing the opposite facet. A \emph{quasi-medial hyperplane} of a $d$-simplex is a hyperplane containing a $(d-2)$-face or ridge, as well as the midpoint of the opposite edge, thereby containing a quasi-median. Note that Averkov \cite{a2003:otgosims} uses different terminology, namely the term medial hyperplane, which we have reserved for hyperplanes bisecting every segment between a vertex and a point in the opposite facet. 

In the previous Sections, we have focused on the circumsphere of a $d$-simplex, which may not always exist. However, every Minkowskian simplex possesses an insphere. In the plane, the center of the incircle is found at the intersection of \emph{Glogovski{\u\i}'s angular bisectors} \cite{g1970:botmpwn}. In higher dimensions, the analogous notion, called \emph{bisector} by Averkov in \cite{a2003:otgosims}, is that of a hyperplane such that each point on the hyperplane has equal Minkowskian distance to two given, intersecting hyperplanes. For each pair of adjacent facets of a simplex, this gives rise to an \emph{interior bisector} (containing interior points of the simplex) and an \emph{exterior bisector} (not containing interior points of the simplex). The intersection of all Minkowskian interior bisectors of a simplex is the \emph{incenter} of the simplex, i.e., the center of the insphere. In \cite{a2003:otgosims} it is proved that there are also up to $d+1$ \emph{exspheres} or \emph{exballs} or \emph{escribed balls} associated to each $d$-simplex, whose centers are found as the intersections of $k$ interior bisectors and $d-k$ exterior bisectors, and whose radii are called \emph{exradii}.

Among the theorems in \cite{a2003:otgosims} we find extended versions of the following two theorems (we recall only statements of interest for our current article). Note that some partial (Minkowskian) statements or Euclidean versions appeared earlier in \cite{cg1985:tfpims, m2000:swef}. 

\begin{thm}\cite[Theorem 6, abbreviated]{a2003:otgosims}\label{thm:4.1}
Let $T$ be a $d$-simplex in a Minkowski space $\mathbb{M}^d(B)$, and $T^*_G$ be the $d$-simplex which is the dual body of $T$ with respect to the centroid $G$ of $T$. Then the following conditions are equivalent.
\begin{enumerate}
\item The Minkowskian heights of $T$ are equal in $\mathbb{M}^d(B)$. 
\item The class of quasi-medial hyperplanes of $T$ coincides with the class of Minkowskian bisectors of $T$ in $\mathbb{M}^d(B)$.
\item The centroid of $T$ coincides with the Minkowskian incenter of $T$ in $\mathbb{M}^d(B)$. 
\item The Minkowskian exradii of $T$ are all equal in $\mathbb{M}^d(B)$.
\item The medians of $T^*_G$ have the same length in $\mathbb{M}^d(B^*)$, where $B^*$ is the dual body to $B$.
\item The centroid of $T^*_G$ coincides with the Minkowskian circumcenter of $T^*_G$ with respect to $\mathbb{M}^d(B^*)$.
\end{enumerate}
\end{thm}

The proof of this theorem uses the fact that the centroids of $T$ and $T^*_G$ coincide, which is proved in \cite[Section 3]{a2003:otgosims} and also in \cite[Section 2]{m2000:swef}. A further specialization of Theorem \ref{thm:4.1} is possible in the planar situation ($d=2$). Recall that
for the unit ball $B$ in the Minkowski plane, the solution to the isoperimetric problem, called \emph{isoperimetrix}, is
the body $\tilde{B}$ obtained from $B$ via rotation of the dual $B^*$ by the angle $\pi/2$ (w.r.t. the auxiliary inner product), see \cite{b1947:tipitmp} and \cite[Chapter 4]{t1996:mg}. Both bodies are homothetic if and only if the plane is a so-called \emph{Radon plane}, cf. \cite{ms2006:aarc}. Radon planes are precisely those planes where Birkhoff orthogonality is symmetric. In general, the norm induced by the isoperimetrix of a unit ball (itself a convex body symmetric to the origin) is called \emph{antinorm}, see \cite{ms2006:aarc}. A convex body is called \emph{reduced} with respect to a given norm of the ambient space if any properly contained convex body has smaller minimum width (for reduced bodies in Minkowski spaces we refer to the survey \cite{lm2014:rcbifdnsas}). 
\begin{thm}\cite[Theorem 9]{a2003:otgosims}\label{thm:4.2}
Let $\mathbb{M}^2(B)$ be an arbitrary Minkowski plane, and $T$ be a triangle in $\mathbb{M}^2(B)$. Then the following conditions are equivalent:
\begin{enumerate}
\item $T$ is reduced in $\mathbb{M}^2(B)$.
\item $T$ has equal Minkowskian heights in $\mathbb{M}^2(B)$.
\item $T$ is equilateral in $\mathbb{M}^2(\tilde{B})$, where $\tilde{B}$ is the isoperimetrix of $B$, i.e., the dual rotated by an angle of $\pi / 2$ about the origin.
\end{enumerate}
\end{thm}

For more results concerning reduced triangles in Minkowski planes, especially their metric properties, see also \cite{ams2013:ortinp}.
We deduce the following statements from Averkov's Theorems, basically starting with the dual and making use of the term $AG$-quasiregularity.

\begin{thm}\label{thm:4.3}
Let $T$ be a $d$-dimensional simplex in a Minkowski space $\mathbb{M}^d(B)$, and $T^*_G$ be the $d$-simplex which is the dual body of $T$ with respect to the centroid $G$ of $T$. Then the following conditions are equivalent.
\begin{enumerate}
\item $T$ is $AG$-quasiregular in $\mathbb{M}^d(B)$.
\item The medians of $T$ have the same length in $\mathbb{M}^d(B)$. 
\item The Minkowskian heights of $T^*_G$ are equal in $\mathbb{M}^d(B^*)$, where $B^*$ is the dual body to $B$.
\item The class of quasi-medial hyperplanes of $T^*_G$ coincides with the class of Minkowskian bisectors of $T^*_G$ in $\mathbb{M}^d(B^*)$.
\item The centroid of $T^*_G$ coincides with the Minkowskian incenter of $T^*_G$ in $\mathbb{M}^d(B^*)$.
\item The Minkowskian exradii of $T^*_G$ are all equal in $\mathbb{M}^d(B^*)$.
\end{enumerate}
\end{thm}

Specializing to the plane, we obtain another theorem extending Theorem \ref{thm:4.2}.

\begin{thm}\label{thm:4.4}
Let $\mathbb{M}^2(B)$ be an arbitrary Minkowski plane, and $T$ be a triangle in $\mathbb{M}^2(B)$. Let $T_m$ be a triangle whose sides are all translates of the medians of $T$, and let $\mathbb{M}^2(\tilde{B})$ be the norm with the isoperimetrix of $B$ as the unit disk (also known as \emph{antinorm}). Then the following statements are equivalent.\\
\begin{enumerate}
\item $T$ is $AG$-quasiregular in $\mathbb{M}^2(B)$.
\item $T_m$ is equilateral in $\mathbb{M}^2(B)$. 
\item $T_m$ is reduced in $\mathbb{M}^2(\tilde{B})$.
\item $T_m$ has equal Minkowskian heights in $\mathbb{M}^2(\tilde{B})$.
\end{enumerate}
Furthermore, the following are equivalent.
\begin{enumerate}
\item $T$ is equilateral in $\mathbb{M}^2(B)$.
\item $T_m$ is $AG$-quasiregular in $\mathbb{M}^2(B)$.
\item $T$ is reduced in $\mathbb{M}^2(\tilde{B})$.
\item $T$ has equal Minkowskian heights in $\mathbb{M}^2(\tilde{B})$.
\end{enumerate}
\end{thm}

\begin{proof}
Clearly, with respect to $\mathbb{M}^2(B)$, $T_m$ is equilateral if and only if the medians of $T$ have the same length, which, by Theorem \ref{thm:4.3}, is equivalent to $T$ being $AG$-quasiregular. However, it is also true that $T_m$ is $AG$-quasiregular if and only if $T$ is equilateral; $(T_m)_m$ is homothetic to $T$ with all side lengths scaled by $\frac{2}{3}$, and thus the medians of $T_m$ have equal length if and only if $T$ is equilateral. This proves the equivalence of the first two statements in each set.
We further recall that the antinorm of the antinorm is the original norm, so the remaining equivalences follow from applying Theorem \ref{thm:4.2} to either $T_m$ or $T$.
\end{proof}

Observe that \emph{other equivalences from Theorem \ref{thm:4.3} apply to the corresponding triangles}. As such, Theorems \ref{thm:4.3} and \ref{thm:4.4} together provide a link between side lengths, reducedness, special incenters, and $AG$-quasiregularity (special circumcenters), as well as excircles and bisectors for several closely related triangles with respect to several norms. For example, a triangle $T$ is $AG$-quasiregular in $\mathbb{M}^2(B)$ if and only if $T_m$ is equilateral in the same norm, and if and only if $T_m$ is reduced in the antinorm $\mathbb{M}^2(\tilde{B})$ with the centroid of $T_m$ being an incenter in the antinorm. Furthermore equivalent to this situation is the reducedness of $T^*_G$ in $\mathbb{M}^2(B^*)$, which is again equivalent to the centroid of $T^*_G$ being an incenter of $T^*_G$ in $\mathbb{M}^2(B^*)$, and both the equilaterality of $T^*_G$ and the $AG$-quasiregularity of $(T^*_G)_m$ in the norm induced by $B$ rotated by $\pi/2$ around the origin as the unit ball. We close with a remark on the situation in a Minkowski plane equipped with a \emph{Radon norm}.

\begin{rem}\label{rem:4.1}
In a Radon plane, a triangle $T$ is $AG$-quasiregular if and only if $T_m$ is equilateral or, equivalently, reduced, as well as if and only if the centroid of $T_m$ is an incenter. Equivalently, $T^*_G$ is reduced, equilateral, with the centroid being an incenter with respect to the dual of the Radon norm (which is also a Radon norm, and which is induced by the unit ball rotated by $\pi/2$ around the origin), and $(T^*_G)_m$ is $AG$-quasiregular in the dual of the Radon norm.
\end{rem}

\section{Concluding remarks and open problems}\label{sec:remarks}

We have seen that, compared to the plane, the situation for circumcenters is more complicated in dimension three and up. Further investigation is warranted, and the methodologically similar study of geometric properties of minimal enclosing balls and all their centers (=Chebyshev sets) seems to be another promising topic (see \cite{ams2012:medcacinpp2} for the two-dimensional situation and \cite{mms2014:csabo} for more general settings). Circumcenters and Chebyshev sets are of interest in adjacent disciplines such as location science and approximation theory. Further on, solutions to questions from Elementary Geometry in normed spaces very often yield a tool and form the first step for attacking problems in the spirit of Discrete and Computational Geometry in such spaces (see, e.g., \cite{ams2012:medcacinpp1,ams2012:medcacinpp2}). And of course it is an interesting task for geometers to generalize notions like orthocentricity (cf. \cite{ag1960:otgomp, ms2007:tfcaoinp, mw2009:oosiscnp}) and regularity of figures (see \cite{ms2011:rtinp}) in absence of an inner product. That is, which figures are special, and what are useful symmetry concepts? Nothing satisfactory is done in this direction, and it is clear that a corresponding hierarchical classification of types of simplices would yield the first step here. One possible question is whether $AG$-quasiregular simplices have further interesting properties in higher dimensions, such as those mentioned for Euclidean simplices in Remark \ref{rem:3.1}. Finally we mention that still for the Euclidean plane there are new generalizations of notions, such as generalized Euler lines in view of so-called circumcenters of mass etc. (see \cite{tt2015:rotcom}), which could also be studied for normed planes and spaces.

\begin{flushleft}
Undine Leopold\\
undine.leopold@gmail.com\\[2ex]

Horst Martini\\
Technische Universit\"at Chemnitz \\
Fakult\"at f\"ur Mathematik\\
D - 09107 Chemnitz\\
Germany\\
horst.martini@mathematik.tu-chemnitz.de
\end{flushleft}

\def\cprime{$'$} \def\cprime{$'$}

\end{document}